\documentclass[12pt]{amsart}
\usepackage{amsmath}
\usepackage{amssymb}
\usepackage{bm} 
\usepackage{stmaryrd}
\usepackage{amsthm}
\usepackage{amscd}
\usepackage{amsfonts}
\usepackage{graphicx}%
\usepackage{fancyhdr}
\usepackage{enumerate}
\usepackage{enumitem} 
\usepackage{mathrsfs}
\usepackage{ulem}
\usepackage{hyperref}
\usepackage[all]{xy}

\theoremstyle{plain} \numberwithin{equation}{subsection}
\newtheorem{theorem}{Theorem}[section]

\newtheorem{lemma}[theorem]{Lemma}
\newtheorem{proposition}[theorem]{Proposition}
\newtheorem{proposition-definition}[theorem]{Proposition-Definition}

\theoremstyle{definition}
\newtheorem{remark}[theorem]{Remark}

\newlist{proofenum}{enumerate}{10}
\setlist[proofenum]{label=(\alph*), leftmargin=*, listparindent={0in}, parsep={.2in}}

\fancyheadoffset{5pt}
\def\mf#1{{ \mathfrak{#1} }}

\def\mb#1{{ \mathbb{#1} }}
\def\mc#1{{ \mathcal{#1} }}
\def\tr#1{{ \textrm{#1} }}

\def\isom{{ \, \stackrel \sim \longrightarrow \, }}

\def\id{{ \tr{id} }}
\def\C{{\mb C}}
\def\N{{ \mb N }}

\def\Z{{ \mb Z }}
\def\str#1{{ \mc O_{#1} }}

\def\Um{{ U^- }}

\def\g{{ \mf g }}
\def\b{{ \mf b }}
\def\bminus{{ \mf b^- }}

\def\n{{ \mf n }}
\def\nm{{ \mf n^- }}

\def\L{{ \mc L }}

\def\dv{{ d \varphi }}
\def\dvi{{ d \varphi^{-1} }}

\def\barg{{ \bar U(\g) }}
\def\barb{{ \bar U(\b) }}
\def\barbm{{ \bar U(\bminus) }}
\def\barn{{ \bar U(\n) }}
\def\barnm{{ \bar U(\nm) }}

\def\bargf#1{{ \bar U_{\leq #1}(\g) }}
\def\barnmf#1{{ \bar U_{\leq #1}(\nm) }}

\def\dom{{ \Lambda^+ }}
\def\pos{{ \Delta^+ }}

\def\ho#1{{ H^0(#1) }}
\def\Ho#1{{ H^0 \big(  #1  \big) }}

\def\V#1{{ V(#1) }}
\def\Vbig#1{{ V \big( #1 \big) }}
\def\Vstar#1{{ V(#1^*) }}

\def\Vss#1{{ V(#1^*)^* }} 
\def\Va#1{{ V^a(#1) }}

\def\Vf#1#2{{ V_{\leq #1}(#2) }}
\def\Vs#1#2{{ V_{< #1}(#2) }}
\def\Vsbig#1#2{{ V_{< #1}\big( #2 \big) }}
\def\hof#1#2{{ H^0_{\geq #1}(#2) }}
\def\hos#1#2{{ H^0_{> #1}(#2) }}
\def\Hof#1#2{{ H^0_{\geq #1}\big( #2 \big) }}

\def\hoa#1{{ H^{0, \, a}(#1) }}

\def\gr{{ \tr{gr} }}

\def\kG{{ k[G] }}

\def\kum{{ k[U^-] }}

\def\kkum{{ k[U^-] \otimes k[U^-] }}

\def\chiw#1{{ \chi_{w_0 #1} }}

\def\hy{{ \tr{hy} }}
\def\hynm{{ \hy(\nm) }}
\def\hygg#1{{ \tr{hy}_{#1}(\g) }}
\def\hynmg#1{{ \tr{hy}_{#1}(\nm) }}

\def\hyg{{ \hy(\g) }}
\def\E#1#2{{ E_{#1}^{(#2)} }}
\def\F#1#2{{ F_{#1}^{(#2)} }}
\def\f#1#2{{ f_{#1}^{(#2)} }}

\def\Fo{{F_0}}
\def\fo{{ f_0 }}
\def\Foo{{ \sw{(\Fo)}1 }}
\def\Fot{{ \sw{(\Fo)}2 }}

\def\flag{{G/B}}

\def\tprho{{ 2(p-1)\rho }}
\def\pN{{ (p-1)N }}
\def\symn{{ S(\n) }}

\def\s{{ \bm s }}
\def\r{{ \mf r }}

\def\Id{{ I_\Delta }}
\def\Idp#1{{ I_\Delta^{#1} }}
\def\Ids{{ \mc I_\Delta }} 
\def\Idsp#1{{ \mc I_\Delta^{#1} }}
\def\Idm{{ I_{\Delta,-} }}
\def\Idmp#1{{ I_{\Delta,-}^{#1} }}

\def\fflag{{ \mc X }}
\def\dfflag{{ \Delta_\fflag }}
\def\cchi#1#2{{ \chi_{#1, \, #2} }} 

\def\Fd#1{{ \mc F^\Delta_{#1} }} 
\def\Fs#1{{ \mc F^{\# \Delta}_{#1} }} 
\def\dG{{ G \times G }}
\def\sG{{ G \rtimes G }} 
\def\ksG{{ k[\sG] }}
\def\kdG{{ k[\dG] }}
\def\sw#1#2{{ #1_{(#2)} }}

\def\dbarg{{ \barg \otimes \barg }} 
 
\def\dbarnm{{ \barnm \otimes \barnm }} 
\def\sbarnm{{ \barnm \# \barnm }} 
 
\def\sbarg{{ \barg \# \barg }} 

\def\grdbarg{{ \gr\big(  \dbarg  \big) }}
\def\grsbarg{{ \gr\big(  \sbarg  \big) }}

\def\shyg{{ \hyg \# \barg }}
\def\shynm{{ \hynm \# \barnm }}

\def\VV#1#2{{ V(#1) \otimes V(#2) }}

\def\VVf#1#2#3{{ \left[  \VV{#1}{#2}  \right]_{\leq #3} }}
\def\VVs#1#2#3{{ \left[  \VV{#1}{#2}  \right]_{< #3} }}
\def\VVa#1#2{{ V^a(#1, #2) }}
\def\VVA#1#2{{ V^a \big( #1, #2 \big) }}
\def\vv#1#2{{ v_{#1, #2} }}
\def\vvt{{ \vv \tprho \tprho }}

\def\vva#1#2{{ v^a_{#1, #2} }}
\def\vvat{{ \vva \gamma \gamma }}

\def\VVat{{ \VVA \gamma \gamma }}

\def\hh#1#2{{ \ho{#1} \otimes \ho{#2} }}
\def\hhf#1#2#3{{ \left[  \hh{#1}{#2}  \right]_{\geq #3} }}

\def\hha#1#2{{ \hoa{ #1, #2 } }}

\def\indhof#1#2#3{{ H^0\Big(  \flag, \L \big( \hof {#3}{#1} \otimes \chi_{-#2} \big)  \Big) }}

\def\indhofwo#1#2#3{{ H^0\Big(  \flag, \L \big( \hof {#3}{#1} \otimes \chi_{w_0#2} \big)  \Big) }}

\def\deta#1#2{{ \eta_{#1, #2} }}
\def\Fdm#1{{ \mc F^{-,\Delta}_{#1} }} 

\def\chiw#1{{ \chi_{w_0 #1} }}

\def\es{{ \tr{es} }}
\def\v#1{{ v_{#1} }}
\def\va#1{{ v_{#1}^a }}
\def\t{{ \bm t }}

\def\idx{{ \mc I_{\Delta, X} }}
\def\idxp#1{{ \mc I_{\Delta, X}^{#1} }}

\def\r#1{{ \alpha_{#1} }} 
\def\chim#1{{ \chi_{-#1} }} 
\def\dw#1#2{{ \xi_{#1}( #2 ) }} 
\def\dwo#1#2{{ \dw{ \omega_{#1} }{ #2 } }} 
\def\dwop#1#2{{ \dw{ (p-1)\omega_{#1} }{ #2 } }} 

\def\hop#1{{ \Ho{ (p-1)\omega_{#1} } }}
\def\hoo#1{{ \ho{\omega_{#1}} }} 
\def\hoof#1#2{{ \hof{#1}{\omega_{#2}} }} 
\def\hoop#1#2{{ \Hof{#1}{(p-1)\omega_{#2}} }}
\def\vs#1{{ v^*_{#1} }}

\def\vlp{{ v_{-\theta}^{\otimes p-1} }}
\def\Vq#1#2{{ \dfrac{ \V{#1} }{ \Vs{#2}{#1} } }} 
\def\VQ#1#2{{ \dfrac{ \Vbig{#1} }{ \Vsbig{#2}{#1} } }} 
\def\Vqw#1#2#3{{ \left(  \Vq{#1}{#3} \right) \otimes \chi_{#2} }} 
\def\VVq#1#2#3{{ \dfrac{ \VV{#1}{#2} }{ \VVs{#1}{#2}{#3}  } }}

\def\hs{{ h^* }}
\def\po{{ (p-1)\omega_1 }}

\def\pl{{ (p-1)\theta }}

\def\fop{{ F_0' }}

\def\gc{{ \g_\C }}
\def\cgc{{ \g_\C[t] }} 
\def\ucgc{{ U( \cgc ) }} 
\def\ucgcg#1{{ U_{#1}( \cgc ) }}
\def\ucgcf#1{{ U_{\leq #1}( \cgc ) }}
\def\barcgz{{ \bar U_\Z( \g[t] ) }}
\def\barcg{{ \bar U( \g[t] ) }}
\def\bargz{{ \bar U_\Z( \g ) }}
\def\ev{{ ev }}
\def\evv#1{{ ev_{#1} }} 
\def\VC#1{{ V^\C(#1) }}
\def\VCc#1#2{{ V^\C_{#1}(#2) }}

\begin{document}

\author{Chuck Hague}
\title{The induced PBW filtration, Frobenius splitting of double flag varieties, and Wahl's conjecture}
\begin{abstract}
Let $G$ be a semisimple algebraic group over an algebraically closed field of positive characteristic $p$. Generalizing the construction of the PBW filtration on Weyl modules for $G$ we construct a $G$-stable filtration on tensor products of Weyl modules which we call the induced PBW filtration. We use this filtration to give some purely representation-theoretic conditions which are equivalent to the existence of a Frobenius splitting of the double flag variety $G/B \times G/B$ that maximally compatibly splits the diagonal. In particular, this gives a sufficient condition for Wahl's conjecture to hold for $G$ and we use this criterion to prove that Wahl's conjecture holds in type $G_2$ for $p \geq 11$.
\end{abstract}
\maketitle

\setcounter{tocdepth}{2} 
\tableofcontents

\section{Introduction}

\subsection{}

Let $k$ be an algebraically closed field of positive characteristic $p$ and let $G$ be a semisimple algebraic group over $k$. Let $B \subseteq G$ be a Borel subgroup. For a dominant weight $\lambda$ let $\V \lambda$ denote the Weyl module for $G$ with highest weight $\lambda$ and let $\ho \lambda = \V{-w_0\lambda}^*$ denote the induced module for $G$ with highest weight $\lambda$ (where $w_0$ is the longest element of the Weyl group of $G$). In particular, $\ho \lambda$ is the global section module for an equivariant line bundle on the flag variety $\flag$.

In a series of papers (\cite{FeDeg}--\cite{FFLPBWZ} and others by the same authors), Feigin, Finkelberg, Fourier and Littelmann have investigated a particular $B$-stable filtration on $\V \lambda$ called the PBW filtration. We denote this filtration by $\Vf n \lambda$. The PBW filtration gives rise to a dual filtration $ \hof n \lambda $ on the global section module $ \ho \lambda $. It was shown in \cite{HaDeg} that the dual PBW filtration computes the degree of vanishing at the identity $eB \in \flag$: that is, we have
\begin{align}
\hof n \lambda &= \{ s \in \ho \lambda : s \tr{ vanishes to degree $\geq n$ at $eB$} \}. 
\end{align}
Using this, we gave in \cite{HaDeg} a purely representation-theoretic criterion involving the PBW filtration which is equivalent to the existence of a Frobenius splitting of $\flag$ that maximally compatibly splits the identity.

More precisely, let $U$ denote the unipotent radical of $B$ and consider the opposite unipotent radical $\Um \subseteq B^-$. Denote by $\nm$ the Lie algebra of $ \Um $. Let $\barnm$ and $\hynm$ denote the hyperalgebras of $\Um$ and $\nm$ respectively. Then the $\barnm$-module structure on $\V \lambda$ induces a $\hynm$-module structure on the associated graded module $ \Va \lambda $ coming from the PBW filtration, and the representation-theoretic criterion in \cite{HaDeg} is stated in terms of this action.

The goal of this paper is to generalize the above in the following way. Let $\lambda$, $\mu$ be two dominant weights. We construct a $G$-stable filtration $ \VVf \lambda \mu n $ on the tensor product module $ \VV \lambda \mu $ which we call the \textbf{induced PBW filtration}. We then use this filtration to analyze splittings of the double flag variety $\fflag := \flag \times \flag$, in analogy to the way that the PBW filtration can be used to analyze splittings of $\flag$.

In more detail, dualizing the induced PBW filtration gives a $G$-stable filtration $ \hhf \lambda \mu n $ on the tensor product module $\hh \lambda \mu$. Let $\chi_{ w_0 \mu }$ denote the 1-dimensional $B$-module corresponding to the character $w_0 \mu$. By a result due to Kumar in \cite{KuWahl} we obtain (cf. Proposition \ref{pr:induction}) a $G$-equivariant isomorphism
\begin{equation}
\tr{Ind}_B^G \big(  \hof n \lambda \otimes \chi_{ w_0 \mu }  \big) \cong \hhf \lambda \mu n,
\end{equation}
which explains the adjective "induced."

Furthermore, denote by $ \dfflag \subseteq \fflag = G/B \times G/B $ the diagonal. Then $ \hh \lambda \mu $ is the global section module for an equivariant line bundle on $\fflag$, and in analogy to the dual PBW filtration $\hof n \lambda$ we show in Proposition \ref{pr:induction} that the filtration $\hhf \lambda \mu n$ computes the degree of vanishing along $ \dfflag $. That is, we have 
\begin{align}
\hhf \lambda \mu n &= \{ s \in \hh \lambda \mu : s \tr{ vanishes to degree $\geq n$ along $\dfflag$} \}. 
\end{align}
Continuing the analogy to the PBW filtration, in our main theorem (Theorem \ref{th:Wahl's}) we give a series of purely representation-theoretic statements involving the PBW filtration and the induced PBW filtration which are equivalent to the existence of a Frobenius splitting of $\fflag$ that maximally compatibly splits $\dfflag$. We also show that the existence of such a splitting of $\fflag$ implies the existence of a splitting of $\flag$ that maximally compatibly splits $eB$.

Let us say more about the induced PBW filtration and our main theorem. $\VV \lambda \mu$ is a module for the smash product algebra $ \barnm \# \barnm $ obtained through the adjoint action of $\barnm$ on itself, and this module structure induces a $\shynm$-module structure on the associated graded module $\VVa \lambda \mu$. Using this module structure we obtain multiple representation-theoretic criteria for the existence of such a maximal splitting of $\dfflag$ in $\fflag$. We note that two of these criteria (criteria (3) and (4) in Theorem \ref{th:Wahl's}) do not involve the induced PBW filtration or the ring $\shynm$ (they only involve the PBW filtration), but the proof does use the induced PBW filtration. 

The existence of a splitting of $\fflag$ that maximally compatibly splits the diagonal $\dfflag$ is quite desirable because if such a splitting exists then by \cite{LMP} Wahl's conjecture holds for $\flag$ and hence for all partial flag varieties $G/P$ as well (see \S\ref{subsub:Wahl's} below for details on Wahl's conjecture). By \cite{KuWahl}, \cite{LTmax} and \cite{TWahl} Wahl's conjecture is known to hold in characteristic 0 for all semisimple algebraic groups and in positive characteristic for types $A$ and $C$. In \S\ref{sec:G2} below we use the representation-theoretic criterion in Theorem \ref{th:Wahl's} to show that Wahl's conjecture holds in type $G_2$ for $p \geq 11$. Remark that the $G_2$ construction uses criterion (4) of Theorem \ref{th:Wahl's}; in particular, the construction uses only the PBW filtration and not the induced PBW filtration. As noted above, though, the proof of this criterion does use the induced PBW filtration.

We remark that the graded module $\VVa \lambda \mu$ is the fusion product of two associated pullback modules for the hyper current algebra associated to $G$. In \S\ref{sub:fusion} we give a brief outline of the fusion product construction and how it relates to the module $\VVa \lambda \mu$.

I would like to thank Evgeny Feigin for pointing out the connection between the induced PBW filtration and the fusion product. I would also like to thank Ghislain Fourier for helpful comments.

Also, Jesper Funch Thomsen informed me that in unpublished work he has given a direct construction of a splitting of $\fflag$ that maximally compatibly splits $\dfflag$ in type $G_2$.

\subsection{Setup} \label{sub:setup}

As above let $G$ be a semisimple algebraic group over $k$ with Borel $B \subseteq G$. Let $U \subseteq B$ be its unipotent radical. Let $B^- \subseteq G$ be the opposite Borel subgroup to $B$ and let $U^- \subseteq B^-$ be its unipotent radical. Let $T \subseteq B$ be a maximal torus. Set $\g = \tr{Lie}(G)$, $\n = \tr{Lie}(U)$ and $\nm = \tr{Lie}(U^-)$. 

Let $\Lambda$ denote the weight lattice of $G$ and let $\pos \subseteq \Lambda$ denote the positive roots corresponding to our choice of Borel $B$. Set $N := |\pos|$ and denote by $\rho$ the half-sum of the positive roots. Denote by $\dom \subseteq \Lambda$ the set of dominant weights. For $\lambda \in \dom$ set $\lambda^* := -w_0 \lambda \in \dom$, where $w_0$ is the longest element of the Weyl group of $G$. For any $\lambda \in \Lambda$ we have the one-dimensional $B$-module $\chi_\lambda$ with weight $\lambda$. 

For any $B$-module $M$ let $\L(M)$ denote the $G$-equivariant bundle on $\flag$ with fiber $M$. For $\lambda \in \dom$ set $ \L(\lambda) := \L(\chiw \lambda) $. Then we have the induced module
$$ \ho \lambda := \Ho{ \flag, \L(\lambda) } $$
with highest weight $\lambda$. Dually, we have the Weyl module
$$ V(\lambda) := \ho{\lambda^*}^* $$
with highest weight $\lambda$.
For each $\lambda \in \dom$ fix a nonzero highest-weight element $v_\lambda \in \V \lambda$. For $\lambda, \mu \in \dom$ we have the tensor product $ \VV \lambda \mu $. Set
$$ \vv \lambda \mu := v_\lambda \otimes v_\mu \in \VV \lambda \mu . $$

For each $\lambda$ pick nonzero $a_\lambda \in \chi_\lambda$. To ease some potentially cumbersome notation, in the sequel we will frequently use the following slight abuse of notation: For a $B$-module morphism $f : V \to W$ and any $\lambda \in \Lambda$ we will also denote by $f$ the twisted morphism $ V \otimes \chi_\lambda \to W \otimes \chi_\lambda $. In this context, given $v \in V$, when we write $f(v)$ we mean $ (f \otimes \id_{\chi_\lambda})( v \otimes a_\lambda ) \in W \otimes \chi_\lambda $.

\subsection{Hyperalgebraic setup} \label{sub:hyperalgebras}

Denote by $ \barg $, $\barn$, $\barb$, $\barnm$ and $\barbm$ the hyperalgebras of $G$, $U$, $B$, $U^-$, and $B^-$, respectively. Let $\bargf n$ denote the degree filtration on $\barg$. This filtration restricts to a degree filtration $\barnmf n$ on $\barnm$. Let $\hyg$ and $\hynm$ denote the hyperalgebras of $ \g $ and $\nm$, respectively, where we consider these as algebraic $k$-groups isomorphic to some number of copies of $\mb G_a$. For $n \geq 0$ let $\hygg n$, $\hynmg n$ be the degree-$n$ graded parts of these algebras. Note that the adjoint action of $G$ on $\g$ (resp. $B^-$ on $\nm$) gives rise to a $G$-algebra structure on $\hyg$ (resp. a $B^-$-algebra structure on $\hynm$) and hence also a $\barg$-algebra (resp. $\barbm$-algebra) structure.

Let $\Delta$, $\sigma$ denote the comultiplication and coinverse in $\barg$. We will use the Sweedler notation
\begin{align*}
\Delta X &= \sum \sw X1 \otimes \sw X2, \\
\big(  (\Delta \otimes \id) \circ \Delta  \big)(X) &= \sum \sw X1 \otimes \sw X2 \otimes \sw X3,
\end{align*}
etc. Denote by $*$ the adjoint action of $\barg$ on itself. Explicitly, this action is given by
$$  X*Y = \sum \sw X1 Y \sigma(\sw X2).  $$

The adjoint $G$-algebra structure on $\barg$ is filtration-preserving and gives rise to a $G$-algebra structure on $\gr \, \barg$. By \cite{FP}, Theorem 2.1, there is a $G$-equivariant algebra isomorphism
\begin{equation} \label{eq:beta 1}
\beta : \gr \, \barg \isom \hyg
\end{equation}
which restricts to a $B^-$-equivariant isomorphism
\begin{equation} \label{eq:beta 2}
\gr \, \barnm \isom \hynm.
\end{equation}

Fix a Chevalley basis $ \{ E_\beta , F_{\beta}: \beta \in \pos \}  \subseteq \g $. Then $\barn$ (resp. $\barnm$) is generated as a $k$-algebra by the divided-power elements $\E \beta n$ (resp. $\F \beta n$) for $\beta \in \pos$ and $n \geq 0$. Denote by $ \f \beta n \in \hynm $ the image of the basis element $\F \beta n$ under the projection $\barnmf n \twoheadrightarrow \hynmg n$. Then $\hynm$ is a divided-power polynomial ring over $k$ on the generators $ \{ \f \beta n : \beta \in \pos, n > 0 \} $. 

Set
\begin{subequations} 
\begin{equation}
\Fo := \prod_{\beta \in \pos} \F \beta{p-1} \in \barnmf \pN,
\end{equation}
a central element in $\barnm$. (We may take any ordering of $\pos$ here, cf. Lemma 6.6 and Proposition 6.7 in \cite{Hab80}.) Projecting $\Fo$ to $\hynm$ we obtain the element
\begin{equation}
\fo := \prod_{\beta \in \pos} \f \beta{p-1} \in \hynmg \pN.
\end{equation}
\end{subequations}

\section{Algebra filtrations and the smash product}

\subsection{The filtration $\Fd n$ on $\dbarg$}

Let $ \Id \subseteq \kdG $ be the ideal of the diagonal $\Delta_G \subseteq \dG$. Since $\dbarg$ identifies with the hyperalgebra of $G \times G$ we may consider $ \dbarg $ as a subspace of the linear dual of $\kdG$. Consider the following two multiplicative filtrations on $\dbarg$ given for $n \geq 0$ by
\begin{subequations} 
\begin{equation} \label{eq:Fd 1}
\{  \mu \in \dbarg : \mu( \Idp {n+1} ) = 0  \}
\end{equation}
and
\begin{equation} \label{eq:Fd 2}
\Delta \barg \cdot ( \bargf n \otimes 1 ) = ( \bargf n \otimes 1 ) \cdot \Delta \barg.
\end{equation}
\end{subequations}
(In (\ref{eq:Fd 2}) we are using the fact that
$$  \Delta X \cdot ( Y \otimes 1 ) = \big(  (X*Y) \otimes 1  \big) \cdot \Delta X  $$
for all $X, Y \in \barg$.)

\begin{proposition-definition} \label{pr:Fd} The filtrations (\ref{eq:Fd 1}) and (\ref{eq:Fd 2}) on $\dbarg$ are the same. Denote this filtration by $\Fd n$. 
\end{proposition-definition}

\begin{proof}
Let us temporarily denote the filtration (\ref{eq:Fd 1}) by $F_n$. Consider the variety automorphism $\varphi$ of $G \times G$ given by
$$  \varphi(a, b) = (ab, b)  $$
for all $a, b \in G$. Since $\varphi$ preserves the identity we obtain a linear automorphism $\dv$ of $\dbarg$ and we have an induced filtration $ \dvi F_n $ on $\dbarg$. Note that $ \varphi(1 \times G) = \Delta_G $. Since the ideal of $ 1 \times G $ in $\kdG$ is $I \otimes \kG$ we have
\begin{align} \label{eq:Fd 3}
\dvi F_n &= \{  \mu \in \dbarg : \mu(I^{n+1} \otimes \kG) = 0 \} \\
\nonumber &= \bargf n \otimes \barg.
\end{align}
Now, we have
$$  \dv(X \otimes Y) = (X \otimes 1) \cdot \Delta Y $$
for all $X, Y \in \barg$. Applying $\dv$ to (\ref{eq:Fd 3}) it follows that
\begin{align*}
F_n &= \dv\big(  \bargf n \otimes \barg  \big) \\
&= \big(  \bargf n \otimes 1 \big) \cdot \Delta \barg
\end{align*}
as desired.
\end{proof}

\begin{remark}
It is not difficult to show more generally that for any $a, b \geq 0$ with $a + b = n$ we have
\begin{align*}
\Fd n &=  [\bargf a \otimes \bargf b] \cdot \Delta \barg \\
&= \Delta \barg \cdot [\bargf a \otimes \bargf b] .
\end{align*}
\end{remark}

\subsection{The smash product}

\subsubsection{}

The adjoint action $*$ of $\barg$ on itself gives rise (c.f. \cite{MoHopf}, Definition 4.1.3) to the smash product ring $ \sbarg $ whose underlying vector space is $\dbarg$ and with multiplication given as follows:
\begin{equation}
A \# X \cdot B \# Y = \sum A \cdot \big(  \sw X 1 * B  \big) \# \sw X 2 Y.
\end{equation}
Remark that we may identify $\sbarg$ with the hyperalgebra of $\ksG$, where the semidirect product structure is given by the conjugation action of $G$ on itself.

There is (c.f. \cite{MoHopf}, Example 4.1.9 and \S7.3.3) a $k$-algebra isomorphism
\begin{align} \label{eq:psi}
\psi : \sbarg &\isom \dbarg, \\
\nonumber A \# B &\mapsto (A \otimes 1) \cdot \Delta B.
\end{align}
In particular we have
\begin{subequations} 
\begin{equation} \label{eq:s left}
\psi\big(  \barg \# 1  \big) = \barg \otimes 1
\end{equation}
and
\begin{equation} \label{eq:s right}
\psi\big( 1 \# \barg \big) = \Delta \barg.
\end{equation}
\end{subequations}

\subsubsection{$\Fs n$ and the associated graded rings} \label{subsub:filtrations}

Transporting the filtration $\Fd n$ to a filtration on $\sbarg$ via the isomorphism $\psi$ of (\ref{eq:psi}) we obtain a multiplicative filtration $\Fs n$ on $\sbarg$. As a result we obtain the associated graded algebras
$$ \grdbarg \cong \grsbarg . $$
By Proposition \ref{pr:Fd}, (\ref{eq:s left}) and (\ref{eq:s right}) we have
\begin{equation}\label{eq:Fs}
\Fs n = \bargf n \# \barg
\end{equation}
for all $n \geq 0$. It follows by (\ref{eq:beta 1}) that there are algebra isomorphisms
\begin{align} \label{eq:gr 1}
\gr \big( \dbarg \big) &\cong \gr \big( \sbarg \big) \\
\nonumber &\cong \big( \gr \, \barg \big) \# \barg \\
\nonumber & \cong \shyg,
\end{align}
where the smash product structure $\shyg$ is obtained from the adjoint $\barg$-algebra structure on $\hyg$.

Note that we will give another construction of the ring $\shyg$ in \S\ref{sub:fusion} below as a quotient of the so-called hyper current algebra.

\subsubsection{Restriction to $\dbarnm$ and $\sbarnm$} \label{subsub:sbarnm}

The above constructions restrict to the subalgebras $ \dbarnm \subseteq \dbarg $ and $ \sbarnm \subseteq \dbarnm $. First, by (\ref{eq:beta 2}) the isomorphism $\psi$ restricts to an isomorphism
\begin{equation} \label{eq:psi 2}
\sbarnm \isom \dbarnm.
\end{equation}
Restricting $ \Fd n $ to $\dbarnm$ we obtain a filtration $ \Fdm n $ on $ \dbarnm $ and, via (\ref{eq:psi 2}), a pullback filtration on $\sbarnm$. Thus we also obtain algebra isomorphisms
\begin{align} \label{eq:gr 2}
\gr\big(  \dbarnm  \big) &\cong \gr \big( \sbarnm \big) \\
\nonumber &\cong \hynm \# \barnm.
\end{align}

\section{The induced PBW filtration}


\subsection{The PBW filtration} \label{sub:PBW}

Following \cite{FFLPBWZ}, for $\lambda \in \dom$ define an increasing filtration on the cyclic $\barnm$-module $V(\lambda)$ by
\begin{equation}
\Vf n \lambda := \barnmf n . v_\lambda = \bargf n .v_\lambda
\end{equation}
(recall that $v_\lambda \in \V \lambda$ is a nonzero highest weight vector). This is called the \textbf{PBW filtration} on $\V \lambda$. The PBW filtration is $B$-stable and gives rise to an associated graded $B$-module
\begin{equation}
\Va \lambda := \gr \,\V \lambda.
\end{equation}
Dually, replacing $\lambda$ by $\lambda^*$ we obtain a decreasing $B$-stable filtration on $ \ho \lambda $ by
\begin{equation} \label{eq:hof}
\hof n \lambda := \left( \Vq{\lambda^*}n \right)^* \subseteq \Vss \lambda = \ho \lambda
\end{equation}
and we obtain an associated graded $B$-module
\begin{equation}
\hoa \lambda := \gr \, \ho \lambda
\end{equation}
which is dual to $\Va{\lambda^*}$.

By \S2, \S3 in \cite{FFLPBWZ} there is a natural $\hynm$-module structure on $ \Va \lambda $ coming from the $\barnm$-module structure on $\V \lambda$. Further, let $ \va \lambda $ be the image of $\v \lambda$ under the natural identification of $ \Vf 0\lambda $ with the highest-weight subspace of $\Va \lambda$. Then $\Va \lambda$ is a cyclic $\hynm$-module generated by $\va \lambda$. There is a natural $B$-algebra structure on $\hynm = \hy( \g / \b )$ such that the $\hynm$- and $B$-module structures on $\Va \lambda$ are compatible.

\subsection{The induced PBW filtration}

\subsubsection{} \label{subsub:induced}

For $\lambda, \mu \in \dom$ consider the Weyl module $\VV \lambda \mu$ for $\dG$. Since $ \VV \lambda \mu $ is a cyclic $\dbarnm$-module we define an exhaustive filtration on this module by
\begin{equation} \label{eq:VVf}
\VVf \lambda \mu n := \Fdm n .( v_\lambda \otimes v_\mu ) = \Fd n .( v_\lambda \otimes v_\mu ) 
\end{equation}
(recall from \S\ref{subsub:sbarnm} above that $\Fdm n$ is the restriction of the filtration $\Fd n$ to $\dbarnm$). We call this filtration the \textbf{induced PBW filtration}. This filtration is $\Delta \barg$-stable and hence this filtration is stable under the standard diagonal $G$-action on $\VV \lambda \mu$.

Taking the associated graded module we obtain a $G$-module
\begin{equation}
\VVa \lambda \mu := \gr [\VV \lambda \mu].
\end{equation}
Then $\VVa \lambda \mu$ is also a module for the rings
$$ \gr\big(  \dbarg  \big) \cong \gr\big(  \sbarg  \big) \cong \shyg $$
and
$$ \gr\big(  \dbarnm  \big) \cong \gr\big(  \sbarnm  \big) \cong \shynm. $$
By (\ref{eq:s right}) and (\ref{eq:gr 1}), the $G$-action on $\VVa \lambda \mu$ agrees with the action induced by the diagonal subalgebra
$$ \barg \cong 1 \# \barg \subseteq \shyg. $$

Recall that $\vv \lambda \mu = v_\lambda \otimes v_\mu$. Note that
\begin{equation}
\VVf \lambda \mu 0 = \Delta \barg.( \vv \lambda \mu ) = \V{\lambda + \mu} \subseteq \VV \lambda \mu
\end{equation}
so that $ \V {\lambda + \mu} $ is naturally a $G$-submodule of $\VVa \lambda \mu$. Furthermore, $\VVa \lambda \mu$ is a cyclic $ \shynm $-module generated by the image of the highest-weight vector $v_{\lambda + \mu}$ under the inclusion $\V{\lambda + \mu} \subseteq \VVa \lambda \mu$.

As we will see in \S\ref{sub:fusion} below, we can also identify $\VVa \lambda \mu$ with a 2-fold fusion product module for the hyper current algebra associated to $G$.

\begin{remark} \label{rem:filtrations}
Identify the $B$-module $\V \lambda \otimes \chi_\mu$ with the $B \times B$-subspace $\V \lambda \otimes v_\mu$ of $\VV \lambda \mu$. It follows from the equality $ \Fd n = \Delta \barg \cdot ( \bargf n \otimes 1 ) $ that for all $\lambda, \mu \in \dom$ the $B$-equivariant inclusion $\V \lambda \otimes \chi_\mu \hookrightarrow \VV \lambda \mu$ is filtration-preserving, where we take the PBW filtration on $\V \lambda$ and the induced PBW filtration on $\VV \lambda \mu$. Taking the associated graded modules we obtain a $B$-equivariant morphism
\begin{equation} \label{eq:V's}
\Va \lambda \otimes \chi_\mu \to \VVa \lambda \mu.
\end{equation}
In general this map will not be injective; consider for example the case $\mu=0$. It is interesting to ask under what conditions on $\lambda$ and $\mu$ this map will be injective.

Remark that the morphism (\ref{eq:V's}) is also $\hynm$-equivariant, where we take the $\hynm$-action on $\VVa \lambda \mu$ induced by the inclusion $\hynm \cong \hynm \# 1 \subseteq \shynm$.
\end{remark}

\subsubsection{The dual induced PBW filtration $\hhf \lambda \mu n$}

Dualizing the induced PBW filtration, we get a decreasing $G$-stable filtration 
\begin{equation} \label{eq:hhf}
\hhf \lambda \mu n := \left(  \VVq{\lambda^*}{\mu^*}n  \right)^*
\end{equation}
on $\hh \lambda \mu$ which induces an associated graded $G$- and $ \shyg $-module
\begin{equation}
\hha \lambda \mu := \gr \, \big(  \hh \lambda \mu  \big)
\end{equation}
which is dual to $ \VVa{\lambda^*}{\mu^*} $.

\begin{remark} \label{rem:tangent bundle}
There is a $B$-equivariant Hopf duality pairing $ \symn \otimes \hynm \to k $ obtained from the identification $\symn = k[\nm]$. For $\lambda \in \dom$ we get a $B$-equivariant degree-preserving inclusion $ \hoa \lambda \hookrightarrow \symn \otimes \chiw \lambda$ by dualizing the natural $B$-module surjection
\begin{align*}
\hynm \otimes \chi_{\lambda^*} &\twoheadrightarrow \Va {\lambda^*}, \\
f &\mapsto f.\va \lambda.
\end{align*}
Using this, one can show that there is a $G$-equivariant degree-preserving inclusion
\begin{equation}
\hha \lambda \mu \hookrightarrow H^0\Big(  \flag, \L \big( \symn \otimes \chiw{(\lambda + \mu)} \big)  \Big)
\end{equation}
for all $\lambda, \mu \in \dom$. Further, the right-hand module in the above equation identifies with the space of global sections of a $G$-equivariant line bundle on the tangent bundle $\mc T$ of $\flag$. Thus we may consider the modules $\hha \lambda \mu$ as sections of line bundles on $\mc T$.
\end{remark}

\subsubsection{Connection to geometry}

Recall that $\fflag = G/B \times G/B$. Denote by $\dfflag \subseteq \fflag$ the diagonal and let $ \Ids \subseteq \mc O_\fflag $ be the ideal sheaf of $\dfflag$. Recall that we set $ \L(\lambda) := \L(\chiw \lambda) $. The following proposition gives a connection between the PBW filtration and the induced PBW filtration.
\begin{proposition} \label{pr:induction}
For all $\lambda, \mu \in \dom$ and $n \geq 0$ there are $G$-module isomorphisms
\begin{align}
\indhofwo \lambda \mu n & \cong H^0\big(  \fflag, \Idsp n \otimes \L(\lambda) \boxtimes \L(\mu)  \big) \\
\nonumber & \cong \hhf \lambda \mu n .
\end{align}
\end{proposition}

\begin{proof}
The first isomorphism is contained in the proof of \cite{KuWahl}, Lemma 2.3 (which is valid in arbitrary characteristic). We now consider the second isomorphism.

For $\lambda, \mu \in \Lambda$ denote by $ \cchi \lambda \mu $ the 1-dimensional $ B \times B $-module associated to the $T \times T$-character $(\lambda, \mu)$. There is a nondegenerate Hopf duality pairing
\begin{equation}
\eta : \barnm \otimes \barnm \otimes \kum \otimes \kum \to k
\end{equation}
which defines the hyperalgebra $\barnm \otimes \barnm$ of $\Um \times \Um$. We have a $\dbarg$-module structure on $ \dbarnm $ obtained from the identification with the zero Verma module
$$  \big( \dbarg \big) \otimes_{\barb \otimes \barb} \chi_{0,0}  $$
for $\dbarg$. Thus, via the pairing $\eta$, we obtain a $\dbarg$-module structure on $\kkum$ and $\eta$ is a $\dbarg$-equivariant pairing. (Remark that $\kkum$ is the dual zero Verma module for the opposite Borel $B^- \times B^-$ of $G \times G$.)

Let
\begin{equation}
\deta \lambda \mu : \hh \lambda \mu \otimes \Vstar \lambda \otimes \Vstar \mu \to k
\end{equation}
be the $\dG$-equivariant duality pairing. Considering $ \kum \otimes \kum $ as a subspace of the linear dual of $ \barnm \otimes \barnm $ under the pairing $\eta$, we have the $\dbarg$-equivariant section restriction inclusion
\begin{align} \label{eq:induction 1}
\hh \lambda \mu &\hookrightarrow \kum \otimes \kum \otimes \cchi{w_0\lambda}{w_0\mu}, \\
\nonumber a \otimes b &\mapsto \big [ X \otimes Y \mapsto \deta \lambda \mu\big(  a \otimes b \otimes X.v_{\lambda^*} \otimes Y.v_{\mu^*}  \big) \big]
\end{align}
for all $X, Y \in \barnm$.
We will consider $\hh \lambda \mu$ as a $\dbarg$-submodule of $\kum \otimes \kum \otimes \cchi{w_0\lambda}{w_0\mu}$ via this inclusion.

Let $\Idm \subseteq \kkum$ denote the ideal of the diagonal. Identifying $\Um \times \Um$ with the big cell in $\fflag$, we have that the restriction of $H^0\big(  \fflag, \Idsp n \otimes \L(\lambda) \boxtimes \L(\mu)  \big)$ to $ \Um \times \Um $ is $\big( \Idmp n  \otimes \cchi{w_0\lambda}{w_0\mu} \big) \cap [\hh \lambda \mu]$. Since this restriction is $\Delta \barg$-equivariant it follows that $ \Idm  \otimes \cchi{w_0\lambda}{w_0\mu} $ is a $\Delta \barg$-submodule of $ \kkum \otimes \cchi{w_0\lambda}{w_0\mu} $. Thus we have reduced to showing that there is an equality
\begin{equation} \label{eq:induction 2}
\big( \Idmp n  \otimes \cchi{w_0\lambda}{w_0\mu} \big) \cap [\hh \lambda \mu] = \hhf \lambda \mu n
\end{equation}
of $\Delta \barg$-submodules (and hence of $G$-modules) of $ \kkum \otimes \cchi{w_0\lambda}{w_0\mu} $.

Now, via the pairing $\eta$ we have
$$  \big(  \Fdm{n-1}  \big)^\perp = \Idmp n.  $$
Thus by (\ref{eq:VVf}) and (\ref{eq:induction 1}) we have
\begin{align*}
\Idmp n \cap [\hh \lambda \mu] &= \left \{ s \in \hh \lambda \mu : \deta \lambda \mu\big(  s \otimes ( \Fdm{n-1} ).(\vs \lambda \otimes \vs \mu)  \big) = 0 \right \} \\
&= \hhf \lambda \mu n
\end{align*}
as desired.
\end{proof}

\subsection{Wahl's conjecture} \label{sub:Wahl's}

\subsubsection{Background} \label{subsub:Wahl's}

We first recall some background on Wahl's conjecture. For any nonsingular $k$-variety $X$ let $\Delta_X \subseteq X \times X$ denote the diagonal and let $ \idx \subseteq \str{X \times X} $ be the ideal sheaf of $\Delta_X$. Then by definition we have the cotangent bundle
\begin{equation}
\Omega_X = \idx / \idxp 2
\end{equation}
of $X$, where we identify $X$ with $\Delta_X$. For two line bundles $\L_1$ and $\L_2$ on $X$, the projection $ \idx \twoheadrightarrow \Omega_X $ induces a sheaf map
\begin{equation} \label{eq:gaussian}
\Ho{ X \times X, \idx \otimes \L_1 \boxtimes \L_2 } \to \Ho{X, \Omega_X \otimes \L_1 \otimes \L_2}.
\end{equation}

Let $P \supseteq B$ be any parabolic subgroup. In \cite{Wahl}, Wahl conjectured that the map (\ref{eq:gaussian}) is surjective in the special case $X = G/P$ for $\L_1$ and $\L_2$ ample. A full proof in characteristic 0 was given by Kumar in \cite{KuWahl}. In positive characteristic Wahl's conjecture is known to hold in types $A_n$ and $C_n$ by work of Lauritzen and Thomsen (\cite{LTmax} and \cite{TWahl}) via the following criterion due to Lakshmibai, Mehta and Parameswaran \cite{LMP}. (We refer the reader to \S4 of \cite{HaDeg} for details on Frobenius splittings and maximal compatible splittings.)

Recall that $N$ is the number of positive roots of $G$ (equivalently, $N = \dim  \flag$) and that $\fflag = \flag \times \flag$. If $\fflag$ has a Frobenius splitting that splits $\dfflag$ with maximal multiplicity $\pN$ then Wahl's conjecture holds for $\flag$ and hence also for any partial flag variety $G/P$, c.f. Proposition 2.4 in \cite{LTmax}. This is the criterion we will use below.

Remark that the existence of such a splitting of $\fflag$ implies that the tangent bundle of $\flag$ is Frobenius split, c.f. Remark 2.13 in \cite{LTmax}.

\subsubsection{The main theorem}

Set $\gamma := \tprho$. Recall that $\vv \gamma \gamma$ denotes a nonzero highest weight vector in $\VV \gamma \gamma$. Let $\vvat \in \VVa \gamma \gamma$ be the corresponding highest-weight vector. Also recall the elements $\Fo \in \barnmf \pN$ and $\fo \in \hynmg \pN$ from \S\ref{sub:hyperalgebras} above and recall the $\shynm$-module structure on $ \VVa \gamma \gamma $.

Define a partial order $ \preceq $ on $\dom$ as follows: we have $ \nu \preceq \mu $ if and only if $ \mu - \nu \in \dom $. The following is our main theorem.

\begin{theorem} \label{th:Wahl's}
The following are equivalent:
\begin{enumerate}
\item There is a splitting of $\fflag$ that maximally compatibly splits $\dfflag$.
\item $(\fo \# \Fo).\vvat \neq 0$ in $\VVat$.
\item There is a $G$-module $V$ and a $B$-module morphism
$$f : \Vqw \gamma \gamma \pN \to V$$
such that $ \Fo.f(v_0) \neq 0 $, where $v_0 \in \Vq \gamma \pN$ is the image of $\Fo.v_\gamma \in \V \gamma$ under the quotient map.
\item There is a $G$-module $W$, a dominant weight $\nu \preceq \gamma$ and a $B$-module morphism
$$h : \Vqw \nu \gamma \pN \to W$$
such that $ \Fo. \big(  v_{ \gamma - \nu } \otimes h(u_0)   \big) $ is nonzero in $ \V{\gamma - \nu} \otimes W $, where $u_0 \in \Vq \nu \pN$ is the image of $\Fo.v_\nu \in \V \nu$ under the quotient map.
\end{enumerate}
In particular, if any of these equivalent conditions hold for $G$, then (a) Wahl's conjecture holds for $G/P$, where $P \subseteq G$ is any parabolic; (b) the tangent bundle of $G/B$ is Frobenius split; and (c) there is a Frobenius splitting of $\flag$ that maximally compatibly splits the identity $eB$.
\end{theorem}

\begin{proof}

First note that consequences (a) and (b) follow from the discussion above. For (c), retain the notation of condition (3) of the theorem. By Proposition 4.5 in \cite{HaDeg} there is a splitting of $\flag$ that maximally compatibly splits $eB$ if and only if $v_0 \neq 0$. Thus, if condition (3) holds, then $v_0 \neq 0$ and there is such a splitting of $\flag$.

We now prove the equivalence of conditions (1) -- (4). We will split the equivalences (1) $\iff$ (2) and (2) $\iff$ (3) into smaller equivalences. (3) $\iff$ (4) is proved directly.

\medskip

\hspace{-.2in}$\bm{(1) \iff (2):}$
\begin{proofenum}
\item \textbf{$(\fo \# \Fo).\vvat \neq 0$ in $\VVat$ if and only if the image of $\Fo.v_\gamma \otimes \Fo.v_\gamma$ under the projection
$$ q : \VVf \gamma \gamma \pN \twoheadrightarrow \VVa \gamma \gamma_\pN  $$
is nonzero.}

Note that in $\dbarnm$ we have
\begin{equation} \label{eq:Fo}
(\Fo \otimes 1) \cdot \Delta \Fo = \Fo \otimes \Fo
\end{equation}
as $\Fo$ is a norm form for the "small" hyperalgebra of $\barnm$ generated by $\{F_\beta\}_{\beta \in \pos}$. Hence
\begin{equation*}
(\Fo \# \Fo).(\vvt) = \Fo.v_\gamma \otimes \Fo.v_\gamma \in \VVf \gamma \gamma \pN
\end{equation*}
and the statement follows since
$$  q\big(  (\Fo \# \Fo).(\vvt)  \big) = (\fo \# \Fo).\vva \gamma \gamma.  $$

\item \textbf{The image of $\Fo.v_\gamma \otimes \Fo.v_\gamma$ under the projection $q$ is nonzero if and only if there is $\sigma \in \hhf \gamma \gamma \pN$ such that
\begin{equation} \label{eq:gamma pairing}
\deta \gamma \gamma \big(  \sigma \otimes \Fo.v_\gamma \otimes \Fo.v_\gamma  \big) \neq 0.
\end{equation}}
(Recall that 
$$ \deta \gamma \gamma : \hh \gamma \gamma \otimes \VV \gamma \gamma \to k  $$
is the $G \times G$-equivariant pairing; here we use the fact that $\gamma^* = \gamma$.)

This is clear, since
$$  \hhf \gamma \gamma \pN = \left(  \VVq \gamma \gamma \pN  \right)^*  $$

\item \textbf{There is $\sigma \in \hhf \gamma \gamma \pN$ satisfying (\ref{eq:gamma pairing}) if and only if there is a splitting of $\fflag$ that maximally compatibly splits $\dfflag$.}

We have
\begin{equation}
\omega_\fflag^{1-p} \cong \hh \gamma \gamma,
\end{equation}
and this space identifies with the space of Frobenius-linear endomorphisms of the structure sheaf $\mc O_\fflag$ of $\fflag$ (c.f. \S1.3 and \S2.1 in \cite{BK}). Now, the map
$$  \theta : \hh \gamma \gamma \to k $$
that detects Frobenius splittings of $\fflag$ is given by
$$  \theta(\sigma) = \deta \gamma \gamma( \sigma \otimes \Fo.v_\gamma \otimes \Fo.v_\gamma ).  $$ Also, by Proposition \ref{pr:induction} we have that $  \hhf \gamma \gamma \pN $ is the space of global sections of $\omega_\fflag^{1-p}$ that vanish to degree $\geq \pN$ on $\dfflag$. Thus there is a splitting of $\fflag$ that maximally compatibly splits the diagonal if and only if there is $\sigma \in \hhf \gamma \gamma \pN$ such that $\theta(\sigma) \neq 0$, i.e. if and only if (\ref{eq:gamma pairing}) holds.
\end{proofenum}

\medskip

\hspace{-.2in}$\bm{(2) \iff (3):}$

Consider the natural $B \times B$-equivariant inclusion
\begin{align*}
a : \V \gamma \otimes \chi_\gamma &\hookrightarrow \VV \gamma \gamma.
\end{align*}
As noted in Remark \ref{rem:filtrations}, this map is filtration-preserving (where we take the PBW filtration on $\V \gamma$); i.e., we have
\begin{equation*}
a\big(  \Vf n \gamma \otimes \chi_\gamma \big) \subseteq \VVf \gamma \gamma n
\end{equation*}
for all $n \geq 0$. Hence $a$ fits into a commutative diagram
\begin{equation} \label{eq:Wahl's 0}
\xymatrix{
\V \gamma \otimes \chi_\gamma \ar@{^(->}[r]^a \ar@{>>}[d] & \VV \gamma \gamma \ar@{>>}[d]^c \\
\Vqw \gamma \gamma \pN \ar[r]_{\bar a} & \VVq \gamma \gamma \pN
}
\end{equation}
of $B$-modules, where we have denoted by $\bar a$ the induced map on the bottom and by $c$ the quotient map on the right. Note that $ \VVq \gamma \gamma \pN $ is a $G$-module through the diagonal action on $\VV \gamma \gamma$. That is, for $X \in \barg$ and $v \in \VV \gamma \gamma$ we have
\begin{equation} \label{eq:c}
c \big( \Delta X.v) = X.c(v).
\end{equation}

\begin{proofenum} 

\item \textbf{$(\fo \# \Fo).\vvat \neq 0$ in $\VVat$ if and only if $\Fo.\bar a( v_0 ) \neq 0$.}

Considering the $\sbarg$-action on $\VV \gamma \gamma$, we have
$$  (\Fo \# 1).(\vv \gamma \gamma) = a( \Fo.v_\gamma ). $$
Since $\Fo \# 1$ and $1 \# \Fo$ commute in $\sbarg$ it follows that
\begin{equation} \label{eq:Wahl's 1}
(\Fo \# \Fo).(\vv \gamma \gamma) = \Delta \Fo.\big(a( \Fo.v_\gamma ) \big).
\end{equation}
Now, $(\fo \# \Fo).\vvat \neq 0$ in $\VVat$ if and only if the image of $(\Fo \# \Fo).(\vv \gamma \gamma)$ in $\VVq \gamma \gamma \pN$ is nonzero. But by (\ref{eq:c}) and (\ref{eq:Wahl's 1}) this image is $\Fo.\bar a( v_0 )$.

\item \textbf{$\Fo.\bar a( v_0 ) \neq 0$ if and only if statement (3) holds.}

The "only if" part is obvious by taking $f = \bar a$. Conversely, assume given a $G$-module $V$ and morphism $f$ as in statement (3). Denote by
\begin{equation} \label{eq:ev}
ev : \indhof \gamma \gamma \pN \to \hof \pN \gamma \otimes \chi_{-\gamma}
\end{equation}
the standard $B$-module evaluation morphism given by Frobenius reciprocity; then $\bar a$ is the dual map to $ev$. Thus, by Frobenius reciprocity, $f$ extends to a $G$-module morphism $\widehat f$ as in the following commutative diagram:
\begin{equation}
\xymatrix{
\VVq \gamma \gamma \pN \ar[r]^{\hspace{.7in} \widehat f} & V \\
\Vqw \gamma \gamma \pN \ar[u]^{\bar a} \ar[ur]_{ f }
}
\end{equation}
By the assumption on $f$ we have
$$ \widehat f \big(  \Fo.\bar a(v_0)  \big) = \Fo.\big[ \widehat f\big(  \bar a(v_0)  \big) \big] = \Fo.f(v_0) \neq 0  $$
which implies $\Fo.\bar a(v_0) \neq 0$ as desired.

\end{proofenum}

\medskip

\hspace{-.2in} $\bm{(3) \iff (4):}$

The direction (3) $\implies$ (4) is trivial (take $\nu = \gamma$). Conversely, assume given $\nu$, $h$ and $W$ as in the statement. Dualizing the multiplication map
$$ \ho{ \gamma - \nu^* } \otimes \hof \pN {\nu^*} \to \hof \pN \gamma  $$
we obtain a $B$-module map
\begin{equation} \label{eq:Wahl's 2}
d : \Vq \gamma \pN \longrightarrow \V{\gamma - \nu} \otimes \VQ \nu \pN .
\end{equation}
Tensoring (\ref{eq:Wahl's 2}) by $\chi_\gamma$ we obtain a composite map
\begin{equation}
f : \Vqw \gamma \gamma \pN \stackrel d \longrightarrow \V{\gamma - \nu} \otimes \Vqw \nu \gamma \pN \stackrel {\id \otimes h} \longrightarrow \V{\gamma - \nu} \otimes W
\end{equation}
(recall our conventions on twisting by 1-dimensional $B$-modules from \S\ref{sub:setup} above). We now verify that $f$ satisfies the conditions in (3).

As in (3), denote by $v_0 \in \Vq \gamma \pN$ the image of $ \Fo.v_\gamma $ under the quotient map $ \V \gamma \twoheadrightarrow \Vq \gamma \pN $. I claim that
\begin{equation} \label{eq:Wahl's 3}
d(v_0) = v_{\gamma - \nu} \otimes u_0.
\end{equation}
Now, $d$ fits into a commutative diagram
\begin{equation*}
\xymatrix{
\V \gamma \ar@{^(->}[r] \ar@{>>}[d] & \V{\gamma - \nu} \otimes \V \nu \ar@{>>}[d]^{pr} \\
\displaystyle \Vqw \gamma \gamma \pN \ar[r]_{\hspace{-.15in}d} & \V{\gamma - \nu} \otimes \displaystyle \Vq \nu \pN
}
\end{equation*}
so to show the claim it suffices to show that
\begin{equation} \label{eq:Wahl's 4}
pr\big( \Fo.(v_{\gamma - \nu} \otimes v_\nu) \big) = pr\big(  v_{\gamma - \nu} \otimes \Fo.v_\nu  \big).
\end{equation}
Indeed, consider the coexpansion
\begin{equation*}
\Fo.\big(  v_{\gamma - \nu} \otimes v_\nu  \big) = \sum \Foo.v_{\gamma - \nu} \otimes \Fot.v_\nu,
\end{equation*}
where we may assume that $\Foo$, $\Fot$ are monomials in some chosen ordering of $\pos$. If $ \Fot \neq \Fo $ then the degree of $\Fot$ is $< \pN$ and hence the term $\Foo.v_{\gamma - \nu} \otimes \Fot.v_\nu$ is in the kernel of $pr$. Thus (\ref{eq:Wahl's 4}) holds and the claim follows.

Now, by (\ref{eq:Wahl's 3}) we have
\begin{align*}
\Fo.f(v_0) &= \Fo.\Big [\big(  (\id \otimes h) \circ d  \big)( v_0 ) \Big] \\
&= \Fo.\big(  v_{\gamma - \nu} \otimes h(u_0)  \big).
\end{align*}
By assumption this is nonzero and the implication (4) $\implies$ (3) follows.
\end{proof}

\subsection{The fusion product} \label{sub:fusion}

In this section we give a brief outline of the connection between the module $\VVa \lambda \mu$ and the fusion product for hyper current algebras. This will not be used in any other part of the paper, but the connection seems to be of interest so we have included it here. We first describe the fusion product in characteristic 0 and then move to positive characteristic.

None of the material in this section is original; references are \cite{FeKostka} (where the fusion product was first defined over $\C$) and \cite{MoHyperLoop} (for the positive-characteristic theory).

\subsubsection{The fusion product in characteristic 0} \label{subsub:char 0 fusion}

Consider a $\C$-form $\gc$ of $\g$. We have the \textbf{current algebra} $\cgc := \gc \otimes \C[t]$, which is a Lie algebra with bracket given by
\begin{equation}
[X \otimes f, Y \otimes g] = [X, Y] \otimes fg
\end{equation}
for $X, Y \in \gc$ and $f, g \in \C[t]$. For $z \in \C$ there is a natural evaluation morphism $\evv z : \cgc \twoheadrightarrow \g$ via $t \mapsto z$.

For $\lambda \in \dom$ let $\VC \lambda$ be the Weyl (= simple) module for $\gc$ of highest weight $\lambda$. Choose $z \in \C$. Pulling back by the evaluation map $ \evv z $ we obtain a simple $ \cgc $-module which we will denote by $ \VCc z \lambda $. Given $n \geq 0$, $\bm \lambda = (\lambda_1, \cdots, \lambda_n) \in (\dom)^n$ and $\bm z = (z_1, \ldots, z_n) \in \C^n$ we obtain a tensor product $\cgc$-module
\begin{equation}
V^\C_{\bm z}(\bm \lambda) := \VCc{z_1}{\lambda_1} \otimes \cdots \otimes \VCc{z_n}{\lambda_n}
\end{equation}
which is simple if the $z_i$ are pairwise distinct.

Denote by $ \ucgc $ the enveloping algebra of $ \cgc $. The grading on $\cgc$ by polynomial degree induces an algebra grading $\ucgcg n$ on $\ucgc$. Let $U(\gc)$ denote the enveloping algebra of $\gc$; then we have $ U(\gc) = \ucgcg 0 \subseteq \ucgc $ and $ \ucgc $ is naturally a graded $U(\gc)$-algebra. For $m \geq 0$ set
$$  \ucgcf m := \bigoplus_{n=0}^m \ucgcg n.  $$
(Warning: Although this notation is similar to our notation $\bargf m$ introduced in the main part of the paper above, the filtrations are quite different. Do not confuse the polynomial filtration $\ucgcf m$ with the usual degree filtration on an enveloping algebra.)

Let $v \in V^\C_{\bm z}(\bm \lambda)$ be a nonzero highest-weight vector. Then we have a filtration
\begin{equation}
\ucgcf m.v \subseteq V^\C_{\bm z}(\bm \lambda)
\end{equation}
of $V^\C_{\bm z}(\bm \lambda)$ by $U(\gc)$-submodules. We now set
\begin{equation}
\VCc{z_1}{\lambda_1} * \cdots * \VCc{z_n}{\lambda_n} := \gr \, V^\C_{\bm z}(\bm \lambda).
\end{equation}
It is straightforward to check that this is not just a graded $U(\gc)$-module but also a graded $\ucgc$-module. We call this module a \textbf{fusion product} module. Remark that if we forget the grading of this module, it follows by semisimplicity of the finite-dimensional representation category of $ U(\gc) $ in characteristic 0 that the fusion product module is isomorphic as a $U(\gc)$-module to $ V_{\bm \lambda}(\bm a) $. (These modules clearly are not isomorphic as $\ucgc$-modules though.)

Further, the action of $\cgc$ on $\VCc{z_1}{\lambda_1} * \cdots * \VCc{z_n}{\lambda_n}$ factors through the action of the quotient Lie algebra $ \gc \otimes ( \C[t] / t^n ) $. In particular, when $n=2$ we have an induced action by the enveloping algebra $U\big(  \gc \otimes (\C[t]/t^n)  \big)$ of the Lie algebra $\gc \otimes ( k[t] / t^2 )$. This enveloping algebra is naturally isomorphic to the smash product algebra $ \hy(\gc) \# U(\gc) $.

\subsubsection{The fusion product in positive characteristic}

We now give a hyperalgebraic version of the above in positive characteristic. The  difficulty is that we cannot use enveloping algebras so some care must be taken to construct the appropriate hyperalgebraic version of $\ucgc$.

First, there is a $\Z$-lattice $ \barcgz \subseteq \ucgc $ which is the current algebra analog of Kostant's $\Z$-form $\bargz$ of the enveloping algebra; see \cite{MoHyperLoop} for the details of this construction. In particular, the polynomial-degree grading on $ \ucgc $ restricts to a grading on $ \barcgz $ such that $\bargz$ is the degree-0 subalgebra. There is also a natural evaluation morphism
$$  \barcgz \twoheadrightarrow \bargz \otimes_\Z \Z[t].  $$

Recall our base field $k$ of positive characteristic. Base-changing $\barcgz$ to $k$ we obtain an algebra
\begin{equation}
\barcg := \barcgz \otimes_\Z k
\end{equation}
called the \textbf{hyper current algebra} of $G$. $ \barcg $ is graded by polynomial degree and contains $\barg$ as the degree-0 subalgebra. We also have the base-changed evaluation morphism
\begin{subequations} 
\begin{equation}
\ev : \barcg \twoheadrightarrow \barg \otimes k[t].
\end{equation}
For any $z \in k$, composing with polynomial evaluation at $z$ gives a composite map
\begin{equation}
\evv z : \barcg \twoheadrightarrow \barg.
\end{equation}
\end{subequations}
We may now use the evaluation map to pull representations of $\barg$ back to representations of $ \barcg $. In particular, for $\lambda \in \dom$ and $ z \in k $,  pulling back the Weyl module $\V \lambda$ by $\evv z$ gives the module $ V_z(\lambda) $ for $ \barcg $.

For $n \geq 0$, $\bm \lambda = (\lambda_1, \cdots, \lambda_n) \in (\dom)^n$ and $\bm z = (z_1, \ldots, z_n) \in k^n$ we now emulate the characteristic 0 construction and obtain a pullback tensor product module
\begin{equation}
V_{\bm z}(\bm \lambda) := V_{z_1}(\lambda_1) \otimes V_{z_n}(\lambda_n)
\end{equation}
for $\barcg$. Next, we construct a filtration on this module in a manner analogous to the characteristic 0 construction. Taking the associated graded module we obtain the fusion product $\barcg$-module
\begin{equation}
V_{z_1}(\lambda_1) * \cdots * V_{z_n}(\lambda_n) := \gr \, V_{\bm z}(\bm \lambda)
\end{equation}
which is also a graded $\barg$-module. As in the characteristic 0 case, one can check that the $ \barcg $-module structure factors through the action of the quotient algebra $ \bar U( \g[t] / t^n ) $ (where $ \bar U( \g[t] / t^n ) $ denotes an appropriate positive-characteristic hyperalgebraic version of the algebra $U\big(  \gc \otimes (\C[t]/t^n)  \big)$ from \S\ref{subsub:char 0 fusion}.) When $n=2$ this algebra is isomorphic to $ \shyg $. Taking our fusion parameters to be $z_1 = 1$ and $z_2 = 0$, it is easy to check that there is an isomorphism
\begin{equation}
V_1( \lambda ) * V_0(\mu) \cong \VVa \lambda \mu
\end{equation}
of graded $\shyg$-modules.

\section{Computations in $G_2$} \label{sec:G2}

\subsection{Introduction}

We first recall some general background about bases for PBW-graded modules from \cite{FeOrb}. Remark that although our goal is to use these results in type $G_2$, the results in \S\ref{subsub:FeOrb} and \S\ref{subsub:ess} are valid in all types.

\subsubsection{} \label{subsub:FeOrb}

Recall that $N = |\pos|$. Fix an ordering $\beta_1, \ldots, \beta_N$ of $\pos$ such that $ \beta_i < \beta_j $ implies $i > j$. For any $N$-tuple $\s = (s(\beta)) \in \N^\pos$ set
\begin{equation}
F^\s := \prod_{\beta \in \pos} \F \beta{s(\beta)} \in \barnm
\end{equation}
and
\begin{equation}
f^\s := \prod_{\beta \in \pos} \f \beta{s(\beta)} \in \hynm.
\end{equation}
Note that there is an obvious additive structure on $\N^\pos$. Also set
$$ \deg \s := \sum_{\beta \in \pos} s(\beta), $$
the degree of the monomials $f^\s$ and $F^\s$.

Note that section multiplication $ \ho \lambda \otimes \ho \mu \to \ho{\lambda + \mu} $ induces a well-defined degenerate section multiplication map $ \hoa \lambda \otimes \hoa \mu \to \hoa{\lambda + \mu} $. Recall from Remark \ref{rem:tangent bundle} above that for all $\lambda \in \dom$ the cyclic $\hynm$-action on $\Va \lambda$ induces a natural $U$-module inclusion
$$ i_\lambda : \hoa \lambda \hookrightarrow \symn. $$
(Here we are forgetting the torus action so we may drop the twist by $ \chiw \lambda $ which would make this into a $B$-module morphism.) By Hopf duality, degenerate section multiplication is compatible with multiplication in $\symn$ under this map. 

For any $n \geq 0$ we get a composite map
\begin{equation} \label{eq:j}
j^n_\lambda : \hof n \lambda \twoheadrightarrow \hoa \lambda_n \stackrel {i_\lambda} \hookrightarrow \symn
\end{equation}
with kernel $\hos n \lambda$. (See Lemma \ref{lem:FeOrb} below for an explicit description of this map in terms of coordinates.) Compatibility of multiplication implies that for $m,n \geq 0$ and $\lambda, \mu \in \dom$ there is a commutative diagram
\begin{equation} \label{eq:mult}
\xymatrix{
\hof n \lambda \otimes \hof m \mu \ar[r] \ar[d]_{ j^n_\lambda \otimes j^m_\mu } & \hof{ m+n }{\lambda + \mu} \ar[d]^{ j^{m+n}_{\lambda + \mu} } \\
\symn \otimes \symn \ar[r] & \symn,
}
\end{equation}
of $U$-modules, where the horizontal arrows are multiplication.

Consider the monomial basis for $\symn$ dual to our usual monomial basis for $\hynm$ generated by the $\f \beta n$. Equivalently, this is the monomial basis given by the algebra generators $ \{ x_\beta \}_{\beta \in \pos} $ for $\symn$, where $ x_\beta $ is the element dual to $f_\beta \in \hynm$. Note that these generators span a subspace of $\symn$ which is $B$-module isomorphic to $\n$. For $\s \in \N^\pos$ we denote by $x^\s \in \symn$ the obvious monomial.

\subsubsection{Essential bases} \label{subsub:ess}

Recall our fixed ordering of $\pos$ from above. Define a total order on $\N^\pos$ as follows. If $\deg \s < \deg \t$ then we set $\s < \t$. If $\deg \s = \deg \t$ then we use the reverse lexicographic ordering coming from our chosen ordering of $\pos$ given above. (Thus higher powers of lower-height roots are higher in the ordering.) For $\lambda \in \dom$ define the set $\es(\lambda)$ of \textbf{essential multiindices} as follows:
\begin{equation}
\es(\lambda) := \big \{  \s \in \N^\pos : F^\s.v_\lambda \notin \tr{span} \, \{ F^\t.v_\lambda : \t < \s \}  \big \}.
\end{equation}
Since the total order on $\N^\pos$ refines the degree ordering, it follows that the set
$$  \{  f^\s.\va \lambda : \s \in \es(\lambda)  \}  $$
is a basis of $ \Va \lambda $. As a consequence,
$$ \{ F^\s.v_\lambda : \s \in \es(\lambda) \tr{ and deg} \, \s \leq n \} $$
is a basis of $\Vf n \lambda$ for all $n \geq 0$.

Let
$$ \{ \xi_\lambda(\s) : \s \in \es(\lambda^*) \} $$
be the basis of $ \ho \lambda $ dual to the basis $ \{ F^\s.v_{\lambda^*} : \s \in \es(\lambda^*) \} $ of $\V {\lambda^*}$. Then
$$  \{ \xi_\lambda(\s) : \s \in \es(\lambda^*) \tr{ and deg} \, \s \geq n \}  $$
is a basis of $ \hof n \lambda $ for all $n \geq 0$. In the sequel, whenever we write $ \dw \lambda \s $ for some $\lambda \in \dom$ and multiindex $\s$, this will implicitly indicate that $\s \in \es(\lambda)$.

For $\lambda \in \dom$ let $\eta_\lambda : \ho \lambda \otimes \V {\lambda^*} \to k$ be the $G$-equivariant duality pairing. We need the following result from \cite{FeOrb} which we state in terms of the map $j_\lambda^n$ from (\ref{eq:j}).

\begin{lemma}[\cite{FeOrb}, Lemma 2.6] \label{lem:FeOrb}
Choose $\lambda \in \dom$ and $\s \in \es(\lambda^*)$. Then
\begin{equation}
j_\lambda^{\deg \s}\big(  \xi_\lambda(\s)  \big) = \sum \eta_\lambda\big(  \xi_\lambda(\s) \otimes F^\t.v_{\lambda^*}  \big) \cdot x^\t,
\end{equation}
where this sum is taken over all $\t \in \N^\pos$ with $\deg \t = \deg \s$ and $ \t \geq \s $. Further, $x^\t$ appears with coefficient 1 in the above expression, and $\t$ is the only essential multiindex from $\es(\lambda^*)$ that appears with nonzero coefficient.
\end{lemma}

Remark that the main point of the above lemma is that the sum is taken only over $\t \geq \s$. The rest follows from the definitions.

\begin{remark}
Our definition of the elements $\xi_\lambda(\s)$ differs slightly from the definition in \cite{FeOrb} -- here we are using divided powers whereas non-divided powers are used there. In particular, it follows from (\ref{eq:mult}) that in our case the coefficient of $ \xi_{\lambda + \mu}(\s + \t) $ in a dual basis expansion of $ \xi_\lambda(\s) \cdot \xi_\mu(\t) $ is $1$. On the other hand, this coefficient is \emph{not} equal to 1 in \cite{FeOrb}, c.f. the proof of Lemma 2.8 there. In characteristic 0 (which is the setting of \cite{FeOrb}) either definition is fine, and the non-divided power definition is more natural; however, in positive characteristic, we must use the divided-power definition.
\end{remark}

\subsubsection{Bases for Weyl modules in type $G_2$} \label{subsub:ess G2}

For convenience we copy the table of essential multiindices in type $G_2$ (due to Gornitsky \cite{Gor}) from \S3.3 of \cite{FeOrb}. Order the 6 positive roots as follows:
\begin{equation*}
\beta_1 = \r{11122}, \; \beta_2 = \r{1112}, \; \beta_3 = \r{112}, \; \beta_4 = \r{12}, \; \beta_5 = \r2, \; \beta_6 = \r1.
\end{equation*}
Let $\omega_1, \omega_2$ denote the fundamental weights for $G$. For $k, l \geq 0$ set $\lambda := k \omega_1 + l \omega_2$. Then the set $\es(\lambda) \subseteq \N^\pos = \N^6$ of essential multiindices is the set of 6-tuples $\s = (s_i)_{i=1}^6$ given by:
\begin{gather*}
s_5 \leq l, \; s_6 \leq k, \\
s_2 + s_3 + s_6 \leq k + l, \; s_3 + s_4 + s_6 \leq k+l, \; s_4 + s_5 + s_6 \leq k+l, \\
\sum_{i=1}^5 s_i \leq k + 2l, \; \sum_{i=2}^6 s_i \leq k + 2l.
\end{gather*}

We will often use the above table in the sequel without explicit mention. For $\lambda \in \dom$, we will denote dual basis elements $\xi_\lambda(-)$ with arguments from $\es(\lambda)$ considered as a subset of $\N^6$. We denote the algebra generators $x_\beta$ of $\symn$ described above by $ x_{11122} := x_{\alpha_{11122}} $, $ x_{1112} := x_{\alpha_{1112}} $, etc.

\subsection{The construction in type $G_2$}

\subsubsection{Setup}

Assume that $p \geq 11$ and that $G$ is of type $G_2$. Recall that in this case $N = 6$. Set $\theta := 3\omega_1 + \omega_2$ and $\nu := (p-1)( \omega_1 + 2\omega_2 )$. Set $W := \V \theta^{\otimes p-1}$. Recall that $\gamma = \tprho$. In the sequel we will frequently use our convention on twisting $B$-module morphisms by characters, c.f. \S\ref{sub:setup} above. Note that for all $\mu \in \dom$ we have $\mu = \mu^*$.

Here is the idea behind the construction. The goal is to construct a $B$-module morphism
\begin{equation}
h :  \Vqw \nu \gamma \pN \to W
\end{equation}
that satisfies condition (4) in Theorem \ref{th:Wahl's}. We will first construct a map $\hs$ and dualize it to get $h$.

Remark that we take $p \geq 11$ because this implies $\theta$ is a $p$-restricted weight and hence $\V \theta$ is a self-dual simple Weyl module; this will be a necessary component of the construction.

\subsubsection{The map $\hs$} \label{subsub:h}

Recall the ordering of $\pos$ given in \S\ref{subsub:ess G2} above and the maps $j^n_\mu$ from (\ref{eq:j}). As above set $\theta = 3\omega_1 + \omega_2$, $\nu = (p-1)(\omega_1 + 2\omega_2)$ and $\gamma = \tprho$. Denote by $\bm{p-1} \in \N^6 = \N^\pos$ the constant multiindex with all coefficients equal to $p-1$. By the \textbf{weight} of a multiindex $\t$ we will mean the weight of the element $x^\t \in \symn$.

\begin{proposition} \label{pr:hs}
There is a $B$-module morphism
$$  \hs : \V \theta^{\otimes p-1} \to \hof \pN \nu \otimes \chim \gamma $$
such that $\xi_\nu(\bm{p-1})$ occurs with nonzero coefficient in the dual PBW basis expansion of some element in $\tr{im}(\hs)$.
\end{proposition}

\begin{proof}

We split the proof up into a series of steps. 

\begin{itemize}[leftmargin=*]
\item Define the following vectors in $\hoof22$:
	\begin{itemize}
	\item $ a_1 := \dwo2{0,0,1,0,1,0} $, a vector of weight $ -\alpha_1 $; and
	\item $ a_2 := \dwo2{0,1,0,0,1,0} $, a vector of weight $0$.
	\end{itemize}
By Lemma \ref{lem:FeOrb}, we have $ j_{\omega_2}^2(a_1) = x_3x_5 $ and $ j_{\omega_2}^2(a_2) = x_2x_5 $. Indeed, there are no degree-2 multiindices of weight $2 \alpha_1 + 2 \alpha_2$ higher than $(0,0,1,0,1,0)$ in our total order on $\N^6$, and similarly there are no degree-2 multiindices of weight $3 \alpha_1 + 2\alpha_2$ higher than $(0,1,0,0,1,0)$.

\item Note that $\hoo2 = \g$. By weight considerations it is straightforward to check that we have the following identity in $ \hoof22^{\otimes 2}$:
\begin{align*}
\E14.(a_1 \otimes a_2) = \E22.(a_1 \otimes a_2) = \E{12}3.(a_1 \otimes a_2) &= \\
\E{112}4.(a_1 \otimes a_2) = \E{1112}3.(a_1 \otimes a_2) = \E{11122}2.(a_1 \otimes a_2) & = 0.
\end{align*}
As the weight of $a_1 \otimes a_2$ is
$$ -\alpha_1 = -\theta + \omega_1 + 2\omega_2 $$
it follows (c.f. Proposition 4.3.1 in \cite{BK}) that we have a $B$-module map
\begin{align} \label{eq:h'}
\V \theta &\to \hoof22^{\otimes 2} \otimes \chi_{-\omega_1 - 2\omega_2} \,,\\
\nonumber v_{-\theta} &\mapsto a_1 \otimes a_2 \,,
\end{align}
where $v_{-\theta} \in \V \theta$ denotes a nonzero lowest-weight vector.

\item Set
$$ v := \dwop{1}{p-1,0,0,0,0,p-1} \in \hoop{2(p-1)}1 .$$
Then $v$ is a highest-weight vector in $\hop1$. Since $(p-1,0,0,0,0,p-1)$ is the only multiindex in $\es\big( (p-1)\omega_1 \big)$ of degree $ \geq 2(p-1) $ it follows that $\hoop{2(p-1)}1$ is the 1-dimensional highest-weight subspace of $\hop1$ spanned by $v$. That is, $ \hoop{2(p-1)}1 \cong \chi_\po $.

\item I claim that $ j_\po^{2(p-1)}(v) = x_{11122}^{p-1} \, x_1^{p-1}$. Since the highest-weight element $v$ is the $(p-1)^{st}$ power of the highest-weight element $v' := \xi_{\omega_1}(1,0,0,0,0,1) \in \hoof21$, it suffices by multiplicativity to show that $ j_{\omega_1}^2(v') = x_{11122} \, x_1 $. This follows by Lemma \ref{lem:FeOrb}: indeed, there are no degree-2 multiindices of weight $ 4 \alpha_1 + 2 \alpha_2 $ higher than $(1,0,0,0,0,1)$ in our total order on $\N^6$.

\item Taking the $(p-1)^{st}$ tensor power of the map (\ref{eq:h'}) and twisting by the trivial $B$-module
$$ \hoop{2(p-1)}1 \otimes \chim{(p-1)\omega_1} \cong \chi_{(p-1)\omega_1} \otimes \chim{(p-1)\omega_1} \cong k $$
we obtain a $B$-module morphism
\begin{align}
h' : \V \theta &\to \hoop{2(p-1)}1 \otimes \hoof22^{\otimes 2(p-1)} \otimes \chim \gamma \,, \\
\nonumber \vlp &\mapsto v \otimes a_1^{\otimes p-1} \otimes a_2^{\otimes p-1}.
\end{align}

\item Denote by $\hs$ the following composite map:
\begin{equation}
\xymatrix{
\V \theta^{\otimes p-1} \ar[d]_{h'} \\
\hoop{2(p-1)}1 \otimes \hoof22^{\otimes 2(p-1)} \otimes \chim \gamma \ar[d]_m \\
\Hof \pN{ (p-1)(\omega_1 + 2\omega_2) } \otimes \chim \gamma \ar@{=}[d] \\
\hof \pN \nu \otimes \chim \gamma
}
\end{equation}
where $m$ is the standard multiplication map on sections. Set
$$ \underline x := x_{11122}^{p-1} \cdot x_{1112}^{p-1} \cdot x_{112}^{p-1} \cdot x_2^{2(p-1)} \cdot x_1^{p-1} \in \symn. $$
By multiplicativity and the explicit description of the images of $v$, $a_1$ and $a_2$ in $\symn$ we have
\begin{align*}
\big( j_\nu^\pN \circ \hs \big) \big( \vlp \big) = \underline x.
\end{align*}
\item Considering the $\barn$-module action on $\symn$ by differential operators, it is straightforward to check that $ x^{\textbf {p-1}} $ occurs with nonzero coefficient in the monomial expansion of
$$ \big( j_\nu^\pN \circ \hs \big)\big(  \E1{p-1}.\vlp  \big) = \E1{p-1}.\underline x \in \symn .$$
By the table in \S\ref{subsub:ess G2} we have $\bm{p-1} \in \es(\nu)$. It follows by Lemma \ref{lem:FeOrb} that $ \dw \nu {\bm{p-1}} $ occurs with nonzero coefficient in $\hs \big(  \E1{p-1}.\vlp  \big)$, as desired.
\end{itemize}
\end{proof}

\subsubsection{The map $h$}
 
The assumption that $p \geq 11$ implies that $\theta$ is a $p$-restricted weight and hence $\V \theta$ is simple and self-dual.
Dualizing the morphism $\hs$ we obtain a morphism
\begin{equation}
h : \Vqw \nu \gamma \pN \to \V \theta^{\otimes p-1}.
\end{equation}
As in condition (4) of Theorem \ref{th:Wahl's}, denote by $u_0 \in \Vq \nu \pN$ the image of $\Fo.v_\nu$ under the projection $ \V \nu \twoheadrightarrow \Vq \nu \pN $.

By Proposition \ref{pr:hs}, $h(u_0)$ is a nonzero element of $\V \theta^{\otimes p-1}$. Note that $ \nu \preceq \gamma $; in particular, $\gamma - \nu = \po$. To apply Theorem \ref{th:Wahl's} and complete the construction in type $G_2$ we need to verify the following lemma.

\begin{lemma}
We have
\begin{equation*}
\Fo.\big(  v_\po \otimes h(u_0)  \big) \neq 0
\end{equation*}
in $\Vbig \po \otimes \V \theta^{\otimes p-1}$.
\end{lemma}

\begin{proof}
First, note that $h(u_0)$ is an element of $\V \theta^{\otimes p-1}$ of weight
$$\nu = (p-1)( \omega_1 + 2\omega_2 ) = (p-1)\theta - (p-1)\alpha_1 .$$
Since $\nu$ is on the $\alpha_1$-string through $(p-1)\theta$ it follows that the $\nu$-weight space of $\V \theta^{\otimes p-1}$ is 1-dimensional and is contained in the submodule $\Vbig{ \pl }$ of $\V \theta^{\otimes p-1}$. As the $\nu$-weight space of $ \Vbig{ \pl } $ is spanned by $ \F1{p-1}.v_\pl $ we have (up to a nonzero scalar)
\begin{equation*}
h(u_0) = \F1{p-1}.v_\pl \in \Vbig \pl \subseteq \V \theta^{\otimes p-1}.
\end{equation*}
Thus we have reduced to verifying that
\begin{equation} \label{eq:G2 lemma}
\Fo.\big(  v_\po \otimes \F1{p-1}.v_\pl  \big) \neq 0 \tr{ in } \Vbig \po \otimes \Vbig \pl.
\end{equation}

Set
$$ \fop := \prod_{ \beta \in \pos \setminus \{ \alpha_1 \} } \F \beta{p-1}. $$
Then, mod terms of left-hand torus weight different than $\po - (p-1)\alpha_1$, we have
\begin{align*}
\Fo.\big(  v_1 \otimes \F1{p-1}.v_{(p-1)\theta}  \big) &= \F1{p-1}.v_1 \otimes \fop.\F1{p-1}.v_\pl \\
&= \F1{p-1}.v_1 \otimes \Fo.v_\pl.
\end{align*}
By the table in \S\ref{subsub:ess G2} we have that $ \mathbf{p-1} \in \es \big( \pl \big) $. This implies that $\Fo.v_\pl \neq 0$ in $\Vbig \pl$. By that table we also have $\F1{p-1}.v_1 \neq 0$ in $\V \po$ (or just use that $ (p-1)\omega_1(\alpha_1^\vee) = p-1 $). Hence $\F1{p-1}.v_1 \otimes \Fo.v_\pl \neq 0$ and (\ref{eq:G2 lemma}) holds as desired.
\end{proof}

Thus we have proved:
\begin{theorem}
Assume that $p \geq 11$ and that $G$ is of type $G_2$. Then there is a Frobenius splitting of $\fflag$ that maximally compatibly splits $\dfflag$. In particular, Wahl's conjecture holds in this case.
\end{theorem}

\newpage
\bibliographystyle{amsplain}
\bibliography{/Users/Charles/Documents/TeX/thebibliography}

\end{document}